\documentclass{amsart}
\usepackage{graphicx,color}
\usepackage{latexsym}
\usepackage{amsfonts}
\usepackage[all]{xy}
\usepackage{amssymb, mathrsfs, amsfonts, amsmath}
\usepackage{amsbsy}
\usepackage{amsfonts}
\newtheorem*{acknowledgement}{Acknowledgement}

\newtheorem{lemma}{Lemma}
\newtheorem{proposition}{Proposition}

\newtheorem{theorem}{Theorem}

\numberwithin{equation}{section}

\title[GQY manifolds]{Generalized quasi Yamabe gradient Solitons}
\author{Benedito Leandro Neto and Hudson Pina de Oliveira}

\address{Universidade de Federal do Oeste da Bahia, Campus Lu\'is Eduardo Magalh\~aes, Rua Itabuna 1278 Sta. Cruz, CEP 47850000, Bahia, Brazil.}

\email{bleandroneto@gmail.com}

\address{Universidade de Federal de Mato Grosso, campus Araguaia, Avenida Valdron Varjão, 6390, Barra do Garças-MT, CEP 78600-000, Mato Grosso, Brazil.}

\email{hudsonmat@hotmail.com}

\allowdisplaybreaks
\numberwithin{equation}{section}
\numberwithin{theorem}{section}

\date{February 3, 2016}

\begin{document}

\newcommand{\spacing}[1]{\renewcommand{\baselinestretch}{#1}\large\normalsize}
\spacing{1.2}

\maketitle

 \begin{abstract}
We prove that a nontrivial complete generalized quasi Yamabe gradient soliton $(M^{n}, g)$ must be a quasi Yamabe gradient soliton on each connected component of $M$ and that a nontrivial complete locally conformally flat generalized quasi Yamabe gradient soliton has a special warped product structure.
\end{abstract}

\section{Introduction}

A complete Riemannian manifold $(M^{n}, g)$, $n \geq 3$, is a {\it generalized quasi-Einstein manifold}, if there exist three smooth functions $f, \mu$ and $\beta$ on $M$ such that
\begin{eqnarray*}
Ric + \nabla^{2}f-\mu df\otimes df = \beta g,
\end{eqnarray*}
where $Ric$ and $\nabla^{2}$ denotes, respectively, the Ricci tensor and Hessian of the metric $g$. This concept, introduced by Catino in \cite{Catino}, generalizes the $m$-quasi-Einstein manifolds (see, for instance \cite{BarrosRibeiro,Leandro1}). Inspired by \cite{Catino}, we will introduce a class of Riemannian manifolds (see \cite{CMMR}).

A complete Riemannian manifold $(M^n, g)$, $n \geq 3$, is a {\it generalized quasi Yamabe gradient soliton} (GQY manifold), if there exist a constant $\lambda$ and two smooth functions, $f$ and $\mu$, on $M$, such that
\begin{eqnarray}\label{eq0}
(R-\lambda)g=\nabla^{2}f-\mu df\otimes df
\end{eqnarray}
where $R$ denotes the scalar curvature of the metric $g$ and $df$ is the dual $1$-form of $\nabla f$.
In a local coordinates system, we have
\begin{eqnarray}\label{eq1}
(R-\lambda)g_{ij}=\nabla_{i}\nabla_{j}f-\mu\nabla_{i}f\nabla_{j}f.
\end{eqnarray}
When $f$ is a constant function, we say that $(M^n, g)$ is a {\it trivial} generalized quasi Yamabe graient soliton. Otherwise, it will be called {\it nontrivial}.

Let us point out that if $\mu=0$,  (\ref{eq0}) becomes the fundamental equation of gradient Yamabe soliton. For $\lambda=0$ the Yamabe soliton is {{\it steady}, for $\lambda<0$ is {\it expanding} and for $\lambda>0$ is {\it shrinking}. Daskalopoulos and Sesum \cite{DaskalopoulosSesum} proved that locally conformally flat gradient Yamabe solitons with positive sectional curvature are rotationally symmetric. Then in \cite{CSZ}, they proved that a gradient Yamabe soliton admits a warped product structure whitout any additional hypothesis. They also proved that a locally conformally flat gradient Yamabe solitons has a more special warped product structure. Inspired by the Generalized quasi-Einstein metrics (see \cite{Catino, Leandro1}), they started to consider the quasi Yamabe gradient solitons (see \cite{HL,Leandro,Wang}). In \cite{HL}, they introduced the concept of quasi Yamabe gradient soliton and showed that locally conformally flat quasi Yamabe gradient solitons with positive sectional curvature are rotationally symmetric. Moreover, they proved that a compact quasi Yamabe gradient soliton has constant scalar curvature. Leandro \cite{Leandro} investigated the quasi Yamabe gradient solitons on four-dimensional case and proved that half locally conformally flat quasi Yamabe gradient solitons with positive sectional curvature are rotationally symmetric. And he proved that half locally conformally flat gradient Yamabe solitons admit the same warped product structure proved in \cite{CSZ}. Wang \cite{Wang} gave several estimates for the scalar curvature and the potential function of the quasi Yamabe gradient solitons. He also proved that a quasi Yamabe gradient solitons carries a warped product structure. In \cite{CMMR}, they define and study the geometry of gradient Einstein-type manifolds. This metric generalizes the GQY manifolds.

In this paper, we first prove the following result. 

\begin{theorem}\label{theorem3}
Let $(M^{n},\,g)$, $n\geq3$, be a nontrivial complete generalized quasi Yamabe gradient soliton satisfying (\ref{eq0}). Then, $\mu$ must be constant on each connected component of $M$.
\end{theorem}
Catino, Mastrolia, Monticella and Rigoli \cite{CMMR} showed that a complete generalized quasi Yamabe gradient soliton $(M^{n}, g)$ has a warped product structure without any hypothesis over $g$ (we recommend Theorem 5.1 on \cite{CMMR} to reader (see also \cite{CSZ})). 

As a consequence of Theorem \ref{theorem3}, we have

\begin{theorem}\cite{HL}
Let $(M^{n},\,g)$, $n\geq3$, be a nontrivial complete connected generalized quasi Yamabe gradient soliton satisfying (\ref{eq0}) with positive sectional curvature. Then
\begin{itemize}
\item[(a)] if $n=3$, $(M^{n}, g)$ is rotationally symmetric;
\item[(b)] if $n\geq 5$ and $W=0$,  $(M^{n}, g)$ is rotationally symmetric.
\end{itemize}
\end{theorem}

\begin{theorem}\cite{Leandro}
Let $(M^{4},\,g)$ be a nontrivial complete connected half locally conformally flat generalized quasi Yamabe gradient soliton satisfying (\ref{eq0}) with positive sectional curvature. Then, $M^{4}$ is rotationally symmetric.
\end{theorem}

\begin{theorem}\cite{HL}
Let $(M^{n},\,g)$, $n\geq3$, be a nontrivial compact connected generalized quasi Yamabe gradient soliton satisfying (\ref{eq0}). Then, the scalar curvature $R$ of the metric $g$ is constant.
\end{theorem}

From Theorem \ref{theorem3}, we show that a nontrivial complete connected generalized quasi Yamabe gradient soliton admits a warped product structure (see Proposition \ref{prop1}). In the special case when $(M^{n}, g)$ is locally conformally flat, we can say more
about the warped product structure (see \cite{CSZ,CheegerColding,DaskalopoulosSesum,HL,Wang}).

\begin{theorem}\label{theorem4}
Let $(M^{n},\,g)$, $n\geq3$, be a nontrivial complete connected generalized quasi Yamabe gradient soliton satisfying (\ref{eq0}). Suppose $f$ has no critical point and is locally conformally flat, then $(M^{n}, g)$ is the warped product
\begin{eqnarray*}
(\mathbb{R}, dr^{2})\times_{|\nabla u|}(N^{n-1}, \bar{g}_{N})
\end{eqnarray*}
where $u= e^{-\mu f}$, and $(N^{n-1}, \bar{g})$ is a space of constant sectional curvature.
\end{theorem}

Therefore, when $\mu$ is constant on equation (\ref{eq0}), from the above theorems we also have a classification to the gradient Yamabe solitons.

\section{Proof of Theorem \ref{theorem3}}

In this section we first recall some basic facts on tensors that will be useful in the proof of our main results.  We then prove our Theorem \ref{theorem3}. For operators $S,T:\mathcal{H} \to \mathcal{H}$ defined over an $n$-dimensional Hilbert space $\mathcal{H}$, the {\it Hilbert-Schmidt inner product} is defined according to
\begin{equation}
\langle S,T \rangle =\rm tr\big(ST^{\star}\big),
\label{inner}
\end{equation}where $\rm tr$ and $\star$ denote, respectively, the {\it trace} and the {\it adjoint} operation.

For a Riemannian manifold $(M^{n},\,g),$ $n\geq 3,$ the {\it Weyl tensor} $W$ is defined by the following decomposition formula

\begin{eqnarray}\label{weyl}
R_{ijkl}&=&W_{ijkl}+\frac{1}{n-2}\big(R_{ik}g_{jl}+R_{jl}g_{ik}-R_{il}g_{jk}-R_{jk}g_{il}\big)\nonumber\\
&&-\frac{R}{(n-1)(n-2)}\big(g_{jl}g_{ik}-g_{il}g_{jk}\big),
\end{eqnarray} where $R_{ijkl}$ stands for the Riemannian curvature operator. In \cite{CC}, Cao and Chen introduced a covariant 3-tensor $D$ given by
\begin{eqnarray}\label{cotton}
D_{ijk}&=&\frac{1}{n-2}(R_{jk}\nabla_{i}f-R_{ik}\nabla_{j}f)+\frac{1}{(n-1)(n-2)}(R_{il}\nabla^{l}fg_{jk}-R_{jl}\nabla^{l}fg_{ik})\nonumber\\
&-&\frac{R}{(n-1)(n-2)}(\nabla_{i}fg_{jk}-\nabla_{j}fg_{ik}).
\end{eqnarray} 
The tensor $D$ is skew-symmetric in its first two indices and trace-free, i.e.,
\begin{eqnarray*}
D_{ijk} = -D_{jik}\quad\mbox{and}\quad g^{ij}D_{ijk} = g^{ik}D_{ijk} = g^{jk}D_{ijk} = 0.
\end{eqnarray*}

We will show how these two tensors are related.

In order to set the stage for the proof that follows let us recall some equations for any dimension. Moreover, since 
\begin{eqnarray*}
\nabla_{i}|\nabla f|^{2}=2\nabla_{i}\nabla_{j}f\nabla^{j}f,\quad|\nabla f|^{2}= g^{ij}\nabla_{i}f\nabla_{j}f\quad\mbox{and}\quad\Delta f=g^{ij}\nabla_{i}\nabla_{j}f
\end{eqnarray*}
the trace of (\ref{eq1}) is given by
\begin{eqnarray}\label{eq2}
\Delta f-\mu|\nabla f|^{2}=n(R-\lambda)
\end{eqnarray}
and
\begin{eqnarray}\label{eq3}
(R-\lambda)\nabla_{i}f=\frac{1}{2}\nabla_{i}|\nabla f|^{2}-\mu|\nabla f|^{2}\nabla_{i}f.
\end{eqnarray} Taking the covariant derivative of (\ref{eq2}) we get
\begin{eqnarray}\label{eq4}
n\nabla_{i}R=\nabla_{i}(\Delta f)-(\nabla_{i}\mu|\nabla f|^{2}+ \mu\nabla_{i}|\nabla f|^{2}).
\end{eqnarray}

Now, taking the covariant derivative in (\ref{eq1}) we get
\begin{eqnarray}\label{eq20}
\nabla_{i}Rg_{jk}=\nabla_{i}\nabla_{j}\nabla_{k}f - [\nabla_{i}\mu\nabla_{j}f\nabla_{k}f +\mu(\nabla_{i}\nabla_{j}f\nabla_{k}f+ \nabla_{j}f\nabla_{i}\nabla_{k}f)]
\end{eqnarray}
Contracting (\ref{eq20}) over $i$ and $k$, and using the Ricci equation we obtain
\begin{eqnarray*}
\nabla_{j}R=R_{jl}\nabla^{l}f+\nabla_{j}(\Delta f)-\left[g^{ik}\nabla_{i}\mu\nabla_{k}f\nabla_{j}f+\mu\left(\frac{1}{2}\nabla_{j}|\nabla f|^{2}+\Delta f\nabla_{j}f\right)\right]
\end{eqnarray*}
From (\ref{eq2}) and (\ref{eq4}) and the above equation one has
\begin{eqnarray*}
\nabla_{j}R&=&R_{jl}\nabla^{l}f+n\nabla_{j}R+\nabla_{j}\mu|\nabla f|^{2}+\frac{\mu}{2}\nabla_{j}|\nabla f|^{2}\\
&-&g^{ik}\nabla_{i}\mu\nabla_{k}f\nabla_{j}f - n\mu(R-\lambda)\nabla_{j}f-\mu^{2}|\nabla f|^{2}\nabla_{j}f
\end{eqnarray*}
Then, from (\ref{eq3}) we can infer
\begin{eqnarray}\label{eq13}
(n-1)\nabla_{j}R&=& -R_{jl}\nabla^{l}f -|\nabla f|^{2}\nabla_{j}\mu \nonumber\\
&+&[g^{ik}\nabla_{i}\mu\nabla_{k}f+\mu(n-1)(R-\lambda)]\nabla_{j}f.
\end{eqnarray}

\begin{lemma}\label{lemm1}
Let $(M^n, g)$ be an n-dimensional generalized quasi Yamabe gradient soliton satisfying (\ref{eq1}). Then we have:
\begin{eqnarray*}
W_{ijkl}\nabla^{l}f = D_{ijk} + \left(\frac{n-2}{n-1}\right)(\nabla_{i}\mu\nabla_{j}f\nabla_{k}f-\nabla_{j}\mu\nabla_{i}f\nabla_{k}f)
-\left(\frac{|\nabla f|^{2}}{n-1}\right)(\nabla_{i}\mu g_{jk}-\nabla_{j}\mu g_{ik}).
\end{eqnarray*} where $D_{ijk}$ is defined from (\ref{cotton}).
\end{lemma}
\begin{proof} We may use equation (\ref{eq1}) to obtain
\begin{eqnarray*}
\nabla_{i}Rg_{jk}-\nabla_{j}Rg_{ik}&=&\nabla_{i}\nabla_{j}\nabla_{k}f-\nabla_{j}\nabla_{i}\nabla_{k}f
+\mu(\nabla_{i}f\nabla_{j}\nabla_{k}f-\nabla_{j}f\nabla_{i}\nabla_{k}f)\\
&+&(\nabla_{j}\mu\nabla_{i}f\nabla_{k}f-\nabla_{i}\mu\nabla_{j}f\nabla_{k}f).
\end{eqnarray*}
Then, by Ricci identity, we get
\begin{eqnarray*}
\nabla_{i}Rg_{jk}-\nabla_{j}Rg_{ik}&=&R_{ijkl}\nabla^{l}f+\mu(\nabla_{i}f\nabla_{j}\nabla_{k}f-\nabla_{j}f\nabla_{i}\nabla_{k}f)\\
&+&(\nabla_{j}\mu\nabla_{i}f\nabla_{k}f-\nabla_{i}\mu\nabla_{j}f\nabla_{k}f)
\end{eqnarray*}
Now, from (\ref{eq1}) we have
\begin{eqnarray*}
\nabla_{i}Rg_{jk}-\nabla_{j}Rg_{ik}&=&R_{ijkl}\nabla^{l}f+ \mu(R-\lambda)(\nabla_{i}fg_{jk}-\nabla_{j}fg_{ik})\\
&+&(\nabla_{j}\mu\nabla_{i}f\nabla_{k}f-\nabla_{i}\mu\nabla_{j}f\nabla_{k}f).
\end{eqnarray*}
It then follows from (\ref{weyl}) that
\begin{eqnarray}\label{eq60}
\nabla_{i}Rg_{jk}-\nabla_{j}Rg_{ik}&=&W_{ijkl}\nabla^{l}f+\frac{1}{(n-2)}(R_{ik}\nabla_{j}f-R_{jk}\nabla_{i}f)\nonumber\\
&+&\frac{1}{(n-2)}(R_{jl}\nabla^{l}fg_{ik}-R_{il}\nabla^{l}fg_{jk})-\frac{R}{(n-1)(n-2)}(\nabla_{j}fg_{ik}-\nabla_{i}fg_{jk})\nonumber\\
&+&\mu(R-\lambda)(\nabla_{i}fg_{jk}-\nabla_{j}fg_{ik})+(\nabla_{j}\mu\nabla_{i}f\nabla_{k}f-\nabla_{i}\mu\nabla_{j}f\nabla_{k}f).
\end{eqnarray}

From (\ref{eq13}), we obtain
\begin{eqnarray}\label{eq50}
\nabla_{i}Rg_{jk}-\nabla_{j}Rg_{ik}&=&\frac{1}{(n-1)}(R_{jl}\nabla^{l}fg_{ik}-R_{il}\nabla^{l}fg_{jk})
+\frac{|\nabla f|^{2}}{(n-1)}(\nabla_{j}\mu g_{ik}-\nabla_{i}\mu g_{jk})\nonumber\\
&+&\frac{1}{(n-1)}(\nabla_{j}\mu\nabla_{i}f\nabla_{k}f-\nabla_{i}\mu\nabla_{j}f\nabla_{k}f)
+\mu(R-\lambda)(\nabla_{i}fg_{jk}-\nabla_{j}fg_{ik}).
\end{eqnarray}
Combining (\ref{eq60}) and (\ref{eq50}), we finish the proof of Lemma \ref{lemm1}.
\end{proof}

We define the 3-tensor $E$ as follows
\begin{eqnarray}\label{tensorE}
E_{ijk}&=&  \left(\frac{n-2}{n-1}\right)(\nabla_{i}\mu\nabla_{j}f\nabla_{k}f-\nabla_{j}\mu\nabla_{i}f\nabla_{k}f)\nonumber\\
&-&\left(\frac{|\nabla f|^{2}}{n-1}\right)(\nabla_{i}\mu g_{jk}-\nabla_{j}\mu g_{ik}).
\end{eqnarray}
Taking into account this definition, we deduce from Lemma \ref{lemm1} that
\begin{eqnarray}\label{WDE}
W_{ijkl}\nabla^{l}f=D_{ijk}+E_{ijk}.
\end{eqnarray}

{\it Proof of Theorem \ref{theorem3}.}
Since the Weyl tensor and the $3$-tensor $D$ are trace free, i.e. $g^{jk}W_{ijkl}=g^{jk}D_{ijk}=0$ contracting (\ref{WDE}) over $j$ and $k$, we get
\begin{eqnarray}\label{eq200}
g^{jk}E_{ijk}=0.
\end{eqnarray}
On the other hand, from (\ref{tensorE}) we have
\begin{eqnarray}\label{sytem1}
g^{jk}E_{ijk}=-\frac{|\nabla f|^{2}}{n-1}\nabla_{i}\mu-\left(\frac{n-2}{n-1}\right)g(\nabla\mu, \nabla f)\nabla_{i}f.
\end{eqnarray}
Therefore, from (\ref{eq200}) and (\ref{sytem1}) we get
\begin{eqnarray}\label{eq22}
|\nabla f|^{2}\nabla\mu+(n-2)g(\nabla\mu, \nabla f)\nabla f= 0.
\end{eqnarray}
Whence,
\begin{eqnarray*}
|\nabla f|^{2}|\nabla \mu|^{2}+(n-2)g(\nabla\mu, \nabla f)^{2}=0.
\end{eqnarray*}
Since we have a nontrivial GQY manifold, then $\mu$ is constant on each connected component of $M$.
$\hfill\Box$

\section{The warped product structure}

Following the steps in \cite{CSZ}, we can prove that a GQY manifold admits a warped product structure without any additional hypothesis over $M$. From Theorem \ref{theorem3} by using a conformal change of variable on (\ref{eq0}) ($u=e^{-\mu f}$), we get
\begin{eqnarray}\label{ge1111}
\mu u(R-\lambda)g=\nabla^{2}u.
\end{eqnarray} 
Cheeger and Colding \cite{CheegerColding} characterized the warped product structure of (\ref{ge1111}). We will sketch the proof of such warped product structure here for completeness.

Consider the level surface $\Sigma=f^{-1}(c)$ where $c$ is any regular value of the potential function $f$. Suppose that $I$ is an open interval containing $c$ such that $f$ has no critical point. Let $U_{I}=f^{-1}(I)$. Fix a local coordinates system
$$(x_{1}, x_{2},\cdots, x_{n}) = (r, \theta_{2},\cdots, \theta_{n}) $$
in $U_{I}$, where $(\theta_{2},\cdots,\theta_{n})$ is any local coordinates system on the level surface $\Sigma_{c}$,
and indices $a, b, c,\cdots$ range from $2$ to $n$. Then we can express the metric $g$ as
\begin{eqnarray*}
ds^{2} = \frac{1}{|\nabla f|^{2}}df^{2} + g_{ab}(f,\theta)d\theta_{a}d\theta_{b},
\end{eqnarray*}
where $g_{ab}(f, \theta)d\theta_{a}d\theta_{b}$ is the induced metric and $\theta=(\theta_{a},\cdots,\theta_{n})$ is any local coordinates system on $\Sigma_{c}$. From (\ref{eq3})
\begin{eqnarray*}
\frac{1}{2}\nabla_{a}|\nabla f|^{2}=[(R-\lambda)+\mu|\nabla f|^{2}]\nabla_{a}f=0.
\end{eqnarray*}
Since $|\nabla f|^{2}$ is constant on $\Sigma_{c}$, we can make a change of variable
\begin{eqnarray*}
r(x)=\int\frac{df}{|\nabla f|}
\end{eqnarray*}
so that we can express the metric $g$ in $U_{I}$ as 
\begin{eqnarray*}
ds^{2}=dr^{2}+g_{ab}(r,\theta)d\theta_{a}d\theta_{b}.
\end{eqnarray*}
 Let $\nabla r=\frac{\partial}{\partial r}$, then $|\nabla r|=1$ and $\nabla f=f'(r)\frac{\partial}{\partial r}$ on $U_{I}$. Then, 
\begin{eqnarray}\label{eq333}
\nabla_{\partial r}\partial r=0.
\end{eqnarray}

Now, by (\ref{eq333}) and (\ref{eq0}), it follows that
\begin{eqnarray}\label{eq334}
(R-\lambda)=\nabla^{2}f(\partial r,\partial r)-\mu(df\otimes df)(\partial r, \partial r)=f''(r)-\mu(f'(r))^{2}.
\end{eqnarray}
Whence, from Theorem \ref{theorem3} and (\ref{eq334}), we can see that $R$ is also constant on $\Sigma_{c}$. Moreover, since $g(\nabla f, \partial_{a})=0$, from (\ref{eq0}) the second fundamental formula on $\Sigma_{c}$ is given by
\begin{eqnarray}\label{eq555}
h_{ab}=-g(\partial r, \nabla_{a}\partial_{b})=\frac{\nabla_{a}\nabla_{b}f}{|\nabla f|}=\frac{(R-\lambda)}{|\nabla f|}g_{ab}.
\end{eqnarray}
Therefore, from (\ref{eq334}) and (\ref{eq555}) we have
\begin{eqnarray}\label{eq556}
h_{ab}=\frac{f''(r)-\mu(f'(r))^{2}}{f'(r)}g_{ab}.
\end{eqnarray}
From (\ref{eq556}) the mean curvature is given by 
\begin{eqnarray}\label{eq557}
H=(n-1)\frac{f''(r)-\mu(f'(r))^{2}}{f'(r)}
\end{eqnarray}
wich is also constant on $\Sigma_{c}$.

Furthermore, from the second fundamental formula on $\Sigma_{c}$, we have that
\begin{eqnarray}\label{eq558}
h_{ab}=-g(\partial_{r}, \nabla_{a}\partial_{b})=-g(\partial_{r}, \Gamma^{l}_{ab}\partial_{l})=-\Gamma^{1}_{ab}.
\end{eqnarray}
On the other hand,
\begin{eqnarray}\label{eq559}
\Gamma^{1}_{ab}=-\frac{1}{2}g^{11}\frac{\partial}{\partial r}g_{ab}.
\end{eqnarray}
Therefore, from (\ref{eq556}), (\ref{eq558}) and (\ref{eq559}) we get
\begin{eqnarray}\label{eq560}
2\frac{f''(r)-\mu(f'(r))^{2}}{f'(r)}g_{ab}=\frac{\partial}{\partial r}g_{ab}.
\end{eqnarray}
Hence, it follows from (\ref{eq560}) that
\begin{eqnarray*}
g_{ab}(r,\theta)=(f'e^{-\mu f})^{2}g_{ab}(r_{0},\theta),
\end{eqnarray*}
where the level set $\{r=r_{0}\}$ corresponds to $\Sigma_{r_{0}}=f^{-1}(r_{0})$. For any regular value $r_{0}$ of the potential function $f$.

Therefore we can announce the following result analogous to the Proposition 2.1 in \cite{CSZ} (we also recommend \cite{DaskalopoulosSesum, Wang}).

\begin{proposition}\label{prop1}
Let $(M^{n}, g)$ be a nontrivial complete connected generalized quasi Yamabe gradient Yamabe soliton, satisfying the GQY equation (\ref{eq0}), and let $\Sigma_{c} = f^{-1}(c)$ be a regular level surface. Then
\begin{itemize}
\item[(1)] The scalar curvature $R$ and $|\nabla f|^{2}$ are constants on $\Sigma_{c}$.
\item[(2)]The second fundamental form of $\Sigma_{c}$ is given by
\begin{eqnarray}
h_{ab}=\frac{H}{n-1}g_{ab}.
\end{eqnarray}
\item[(3)] The mean curvature $H=(n-1)\frac{(R-\lambda)}{|\nabla f|}$ is constant on $\Sigma_{c}$.
\item[(4)] In any open neighborhood $U^{\beta}_{\alpha}=f^{-1}\big((\alpha, \beta\big)$ of $\Sigma_{c}$ in which f has no critical
points, the GQY metric $g$ can be expressed as
\begin{eqnarray*}
ds^{2} = dr^{2} + (f'(r)e^{-\mu f})^{2}\bar{g}_{ab} 
\end{eqnarray*}
where $(\theta_{2},\cdots, \theta_{n})$ is any local coordinates system on $\Sigma_{c}$ and $\bar{g}(r,\theta) = g_{ab}(r_{0},\theta)d\theta_{a}d\theta_{b}$ is the induced metric on $\Sigma_{c} = r^{-1}(r_{0})$.
\end{itemize}
\end{proposition}

{\it Proof of Theorem \ref{theorem4}.}
Consider the warped product manifold
\begin{eqnarray}\label{product0}
(M^{n}, g)=(I, dr^{2})\times\phi(N^{n-1}, \bar{g}),
\end{eqnarray}
where $ds^{2}=dr^{2}+(\phi)^{2}\bar{g}$. Fix any local coordinates system $\theta = (\theta_{2},\cdots, \theta_{n})$ on
$N^{n-1}$, and choose $(x_{1}, x_{2},\cdots, x_{n}) = (r,\theta_{2},\cdots,\theta_{n})$. Now (see \cite{besse, CSZ, O'neil}) the scalar curvature formulas of $(M^{n}, g)$ and $(N^{n-1}, \bar{g})$ are related by
\begin{eqnarray*}
R=\phi^{-2}\bar{R}-(n-1)(n-2)\left(\frac{\phi'}{\phi}\right)^{2}-2(n-1)\frac{\phi''}{\phi}.
\end{eqnarray*}
Therefore, since $\phi= f'e^{-\mu f}$, from Theorem \ref{theorem3} and Proposition \ref{prop1} we have that $\bar{R}$ does not depend on $\theta$. Then $\bar{R}$ is constant.

Moreover, the Weyl tensor $W$ for an arbitrary warped
product manifold (\ref{product0}) is given by (see \cite{besse, CSZ, O'neil}):
\begin{eqnarray}\label{29}
W_{1a1b}=-\frac{1}{n-2}\bar{R}_{ab}+\frac{\bar{R}}{(n-1)(n-2)}\bar{g}_{ab},
\end{eqnarray}
\begin{eqnarray}\label{30}
W_{1abc}=0,
\end{eqnarray}
and
\begin{eqnarray}\label{31}
W_{abcd}=\phi\bar{W}_{abcd}.
\end{eqnarray}
Where $\bar{W}$ denotes the Weyl tensor of $(N^{n-1},\bar{g})$.
Therefore, since the warped product manifold (\ref{product0}) is locally conformally flat, i.e. $W=0$, from (\ref{29}) and (\ref{31}) we see that $N$ is Einstein and $\bar{W} = 0$. Then, from (\ref{weyl}) we have
\begin{eqnarray*}
\bar{R}_{abcd}=\frac{\bar{R}}{(n-1)(n-2)}(\bar{g}_{bd}\bar{g}_{ac}-\bar{g}_{bd}\bar{g}_{ac}).
\end{eqnarray*}
Since $\bar{R}$ is constant, we get that $\bar{R}_{abcd}$ is also constant. Thus $N$ is a space form.
$\hfill\Box$

\begin{acknowledgement}
The first author is grateful to Ernani Ribeiro Jr for bringing the paper \cite{CMMR} to his attention. We want to thank Professor Xia Changyu for the helpful remarks and discussions. 
\end{acknowledgement}

\end{document}